\documentclass[10.5pt, reqno]{amsart}
\usepackage{amsfonts}
\usepackage{amsmath}
\usepackage{amssymb}
\usepackage{eufrak}
\usepackage{mathrsfs}
\usepackage{bbding}

\numberwithin{equation}{section} 

\input xy
\xyoption{all}
\title{Algebraic cycles on Prym varieties}
\author{Maxim Arap}

\begin{document}

\begin{abstract} This article proposes a generalization of tautological rings introduced by Beauville and Moonen for Jacobians. The main result is that, under certain hypotheses, the special subvarieties of Prym varieties are algebraically equivalent and their classes belong to the tautological ring. 
\end{abstract}

\address{
Department of Mathematics,
Boyd Graduate Studies Research Center,
University of Georgia,
Athens, GA 30602,
USA}
\address{Current address: Department of Mathematics, Johns Hopkins University, Baltimore, MD 21218, USA}
\email{marap@math.jhu.edu}

\newcommand{\bP}{\mathbb{P}}
\newcommand{\bC}{\mathbb{C}}
\newcommand{\mc}{\mathcal}
\newcommand{\ra}{\rightarrow}
\newcommand{\thra}{\twoheadrightarrow}
\newcommand{\mb}{\mathbb}
\newcommand{\mrm}{\mathrm}
\newcommand{\p}{\prime}
\newcommand{\ms}{\mathscr}
\newcommand{\pl}{\partial}
\newcommand{\ti}{\tilde}
\newcommand{\wti}{\widetilde}
\newcommand{\ol}{\overline}
\newcommand{\ul}{\underline}
\newcommand{\Sp}{\mrm{Spec}\,}

\newtheorem{tautring}{Definition}[section]
\newtheorem{prymgensintro}[tautring]{Theorem}
\newtheorem{algeqintro}[tautring]{Theorem}
\newtheorem{classesintro}[tautring]{Theorem}
\newtheorem{RemPTV}[tautring]{Remark}

\newtheorem{generators}{Definition}[section]
\newtheorem{poncl}[generators]{Lemma}
\newtheorem{prymgens}[generators]{Theorem}
\newtheorem{RemNotPolynInXi}[generators]{Remark}
\newtheorem{RemUniquenessOfT}[generators]{Remark}

\newtheorem{algeq}{Theorem}[section]
\newtheorem{iso}[algeq]{Lemma}
\newtheorem{wirtcon}[algeq]{Proposition}
\newtheorem{wirtcon helper 1}[algeq]{Lemma}
\newtheorem{wirtcon helper 2}[algeq]{Lemma}
\newtheorem{principle of connectedness}[algeq]{Proposition}
\newtheorem{flatness criterion}[algeq]{Proposition}
\newtheorem{Dr. Izadi's argument}[algeq]{Proposition}
\newtheorem{remalgeq}[algeq]{Remark}

\newtheorem{classes}{Theorem}[section]

\renewcommand{\thefootnote}{}

\maketitle
 \footnote{\today}

\thispagestyle{empty}

\section[1]{Introduction.}
For a non-singular variety $V$, we let $\mrm{CH}(V)$ denote the Chow ring of $V$ modulo rational equivalence with $\mb{Q}$-coefficients. The quotient $\mrm{CH}(V)/\!\!\sim_{\mrm{alg}}$ modulo algebraic equivalence is denoted by $A(V)$. If $X$ is a moduli scheme, we say that a generic (resp., general) element $x \in X$ has property $\mc{P}$, if $\mc{P}$ holds on the complement of a countable union of closed subschemes of $X$ (resp., on a dense Zariski open subset of $X$). We work over $\mb{C}$, the field of complex numbers.  

Let $X$ be an abelian variety. Besides the intersection product, the ring $\mrm{CH}(X)$ is endowed with Pontryagin product defined by $$x_1 \ast x_2 = m_\ast (p_1^\ast x_1 \cdot p_2^\ast x_2),$$
where $m\colon X \times X \ra X$ is the addition morphism, and $p_j\colon X\times X \ra X$ is the projection onto the $j^{\mrm{th}}$ factor, cf.\! [BL, p.530]. Moreover, the Chow ring  of $X$ is bi-graded,$$\mrm{CH}(X) = \bigoplus_{p,s}\mrm{CH}^p(X)_{(s)}.$$  The $p$-grading is by codimension. The Beauville grading ($s$) is characterized by: $x \in \mrm{CH}^p(X)_{(s)}$ if and only if $k^*x=k^{2p-s}x$ for all $k \in \mb{Z}$, where $k$ also denotes the endomorphism of $X$ given by $x \mapsto kx$, see [Be86]. The $(s)$-component of a cycle $Z$ is denoted by $Z_{(s)}$. There is a Fourier transform  $$\mc{F}_X\colon \mrm{CH}(X)\ra \mrm{CH}(X),$$ which has been defined by Beauville in relation to the Fourier-Mukai transform, see [Be83]. The operations $\ast, \mc{F}_X$ and the bi-grading descend to $A(X)$.  

When $X$ is the Jacobian $J$ of a smooth curve $C$ of genus $g$, we may fix a point $o \in C$ and embed $C$ in $J$ via the Abel map $\varphi\colon x \mapsto \mc{O}_C(x-o)$. The \emph{small tautological ring} $\mrm{taut}(C)$ of $J$ is defined to be the smallest $\mb{Q}$-subalgebra of $\mrm{CH}(J)$ under the intersection product, which contains the class of the image of $C$ under $\varphi$, and is stable under the operations $\ast, \mc{F}_J, k^\ast$ and $k_\ast$ for all $k \in \mb{Z}$. The \emph{big tautological ring} $\mrm{Taut}(C)$ is defined in the same way, except it is required to contain the image of $\varphi_\ast\colon \mrm{CH}(C) \ra \mrm{CH}(J)$, [Mo, Def.3.2, p.487]. The tautological ring for Jacobians was originally defined and studied by Beauville in [Be04] as a $\mb{Q}$-subalgebra $\ms{T}(C)$ of $A(J)$ under the intersection product. In [Be04] it was shown that $\ms{T}(C)$  is generated by the classes $w^1, \ldots, w^{g-1}$, where $w^{g-d} = (1/d!)C^{\ast d}$. The generators for the tautological rings $\mrm{taut}(C)$ and $\mrm{Taut}(C)$ of the Jacobian $J$  have also been determined in [Po, Thm.0.2, p.461] and [Mo, Thm.3.6, p.489], respectively. The rings $\mrm{taut}(C)$ and $\mrm{Taut}(C)$ have the same image, namely $\ms{T}(C)$, in $A(J)$.  

The following definition generalizes the notions of various tautological rings for a Jacobian. The original idea of considering pairs is due to R. Varley.
\begin{tautring} \label{tautring}
Let $X$ be an abelian variety and let $V\subset X$ be a subvariety. The \emph{small} and the \emph{big} \emph{tautological rings} $\mrm{taut}(X, V)$ and $\mrm{Taut}(X, V)$, respectively, of the pair $(X, V)$ are the smallest subrings of $\mrm{CH}(X)$ under the intersection product, which contain $[V]$ and $\mrm{CH}(V)$, respectively, and are stable under the operations $\ast, \mc{F}_X, k^\ast$ and $k_\ast$ for all $k \in \mb{Z}$. The image of $\mrm{Taut}(X, V)$ in $A(X)$ is called the \emph{tautological ring} of $(X, V)$ and is denoted by $\ms{T}(X, V)$.
\end{tautring}

Our definition of the tautological rings: $\mrm{taut}(J, \varphi(C)), \mrm{Taut}(J, \varphi(C)) \textrm{ and } \ms{T}(J, \varphi(C))$ coincides with the previous definitions denoted in [Mo] by $\mrm{taut}(C), \mrm{Taut}(C)$ and $\ms{T}(C)$, respectively. 

Let $\ti{C} \ra C$ be a degree 2 morphism, which is either \'etale or ramified at two points, and let $\ti{J}$ be the Jacobian of $\ti{C}$. By [Mu74], the connected component of the identity in $\ker(\mrm{Nm}\colon \ti{J} \ra J)$ is a principally polarized abelian variety $(P, \xi)$, called the Prym variety of  $\ti{C} \ra C$. If we fix $\ti{o} \in \ti{C}$, there is a morphism $\psi\colon \ti{C} \ra P$, called the \emph{Abel-Prym map}, which is obtained by composing $\ti{C} \ra \ti{J}, \ti{x} \mapsto \mc{O}_{\ti{C}}(\ti{x} -\ti{o})$ with $1-\iota\colon \ti{J} \ra \ti{J}$, where $\iota$ is the involution induced by the involution on $\ti{C}$ exchanging the sheets of the cover. Consider the cycles $\zeta_n = \mc{F}_P\big([\psi(\ti{C})]_{(n-1)}\big)\in A(P)$. In Section \ref{SecOnGens} we show the following.

\begin{prymgensintro} \label{prymgensintro}
The tautological ring $\ms{T}(P, \psi(\ti{C}))$ is generated as a $\mb{Q}$-subalgebra of $A(P)$ under the intersection product by the cycles $\zeta_n$, where $1\le n \le \dim P-1$ and $n$ is odd.
\end{prymgensintro}

In Sections \ref{SecOnAlgEq} and \ref{SecOnClasses} we study the special subvarieties $V_0$ and $V_1$ of $P$ associated to a complete $g^r_d$ on $C$ (see Section \ref{Special subvarieties} for the definition of $V_0,V_1$). The Brill-Noether variety $W^r_d(C)$ parametrizes invertible sheaves $L$ on $C$ with $\deg L = d$ and $h^0(L)\ge r+1$, see [ACGH, p.153]. The expected dimension of $W^r_d(C)$ is the Brill-Noether number $\rho(g,r,d) = g-(r+1)(g-d+r).$ In Section \ref{SecOnAlgEq} we prove the following: 

\begin{algeqintro}
Assume that $W^r_d(C)$ is reduced and of dimension $\rho(g,r,d)$. If $\rho(g,r,d)>0$, then $V_0$ and $V_1$ are algebraically equivalent.
\end{algeqintro} 

Note that the inequalities $d < 2g$ and $0 < 2r < d$ are needed for the special subvarieties to be defined and homologically equivalent, respectively (see Section \ref{Special subvarieties} for details). 

As a consequence of Theorem \ref{prymgensintro}, the tautological ring $\ms{T}(P, \psi(\ti{C}))$ is spanned as a $\mb{Q}$-vector space by cycles of the form 
$$[\psi(\ti{C})]_{(n_1)} \ast [\psi(\ti{C})]_{(n_2)} \ast \cdots \ast [\psi(\ti{C})]_{(n_r)},$$ where $n_i, r \ge 0$ are varying integers. In Section \ref{SecOnClasses} we show that the class $[V]$ of $V=V_0\cup V_1$ belongs to $\ms{T}(P, \psi(\ti{C}))\subset A(P)$ and express $[V]$ in terms of the generators of $\ms{T}(P, \psi(\ti{C}))$:

\begin{classesintro} 
If $0 < 2r < d < 2g$, then the component of the class $[V]$ in $A^{g-r-1}(P)_{(t)}$ is given by the formula $[V]_{(t)} = c_{t,r,d}\big([\psi(\ti{C})]^{\ast r}\big)_{(t)}$, where $c_{t,r,d}$ are certain rational numbers defined in Section \ref{SecOnClasses}. 
\end{classesintro}

\begin{RemPTV}
\emph{The pairs: a Jacobian with an Abel curve and a Prym variety with an Abel-Prym curve are special cases of a pair $(X, \tau(C))$, where $X \subset J$ is a Prym-Tyurin variety defined by the endomorphism $\sigma$ of $J$ satisfying certain properties and $\tau\colon C \ra X$ is the composition of an Abel map $C \ra J$ and $(1-\sigma)\colon J \ra X$, see [BL, p.369] for details. Moreover, every principally polarized abelian variety is a Prym-Tyurin variety, although not in a unique way, [BL, Cor.12.2.4, p.371].}
\end{RemPTV}

\subsection*{Notation and conventions} The word ``point'' refers to a closed point. $\pi\colon \ti{C} \ra C$ denotes a connected double cover, \'etale or ramified at exactly two points, of a smooth curve $C$ of genus $g$. $(P, \xi)$ denotes the principally polarized Prym variety of dimension $p \ge 2$ associated to $\pi\colon \ti{C} \ra C$, where $\xi \in A^1(P)$ is the class of a theta divisor. If $\pi$ is \'etale, $p=g-1$, and, if $\pi$ is ramified at two points, $p=g$. $\ti{J}$ and $J$ denote the Jacobians of $\ti{C}$ and $C$. We fix $\ti{o} \in \ti{C}$ and let $\ti{\varphi}\colon \ti{C} \ra \ti{J}$ be the Abel map $x\mapsto \mc{O}_{\ti{C}}(x-\ti{o})$. $\iota$ denotes the involution on $\ti{C}$ such that $\ti{C}/\iota = C$ and also the induced involution on $\ti{J}$. We let $u\colon \ti{J} \ra P$ be the restriction of $(1-\iota)$ on the image and $\psi:= u\circ \ti{\varphi}\colon \ti{C}\ra P$ be the Abel-Prym map.  We take $o=\pi(\ti{o})$ and let $\varphi\colon C \ra J$ be the Abel map $x\mapsto \mc{O}_{C}(x-o)$.  We take the Fourier transform $\mc{F}_P\colon A(P)\ra A(P)$  to be $\mc{F}_P(x) =  p_{2\ast} (p_1^\ast x \cdot e^\ell)$, where $\ell := p_1^*\xi + p_2^*\xi - m^*\xi$ is the class of a Poincar\'e line bundle on $P \times P$, $m\colon P\times P \ra P$ is the addition morphism and $p_1, p_2$ are the projections from $P \times P$ onto $P$. The Brill-Noether number is denoted by $\rho(g,r,d) = g-(r+1)(g-d+r)$.

\section{Generators for the tautological ring of a Prym variety.} \label{SecOnGens} 

Let  $[\psi(\ti{C})] = [\psi(\ti{C})]_{(0)}+\cdots+[\psi(\ti{C})]_{(p-1)}$ be the decomposition of $[\psi(\ti{C})]$ into homogeneous components for the Beauville grading. 

\begin{generators}
For each $1 \le n \le p=\dim P$ define the cycle $\zeta_n=\mc{F}_P\big([\psi(\ti{C})]_{(n-1)}\big).$ 
\end{generators}

The main theorem in [Be04] gives a set of generators for the tautological ring of the Jacobian of a smooth connected curve. The following theorem provides an analogous set of generators for $\ms{T}(P, \psi(\ti{C}))$. 

\begin{prymgens} \label{prymgens}
The tautological ring $\ms{T}(P, \psi(\ti{C}))$ is generated as a $\mb{Q}$-subalgebra of $A(P)$ under the intersection product by the cycles $\zeta_n$, where $1\le n \le p-1$ and $n$ is odd. 
\end{prymgens}
\begin{proof} Note that $\zeta_p \in A^p(P)_{(p-1)}=\{0\}$, hence $\zeta_p=0$. Also, since $\psi(\ti{C})$ has symmetric translates, $[\psi(\ti{C})]_{(n)}=0$ for all odd $n$, and therefore, $\zeta_n=0$ for all even $n$. The rest of the proof proceeds as in [Be04, Sec.4]. Consider the $\mb{Q}$-vector subspace $\ms{T}^\p(P)$ of $A(P)$ spanned by the cycles of the form $\zeta_{n_1} \cdot \zeta_{n_2} \cdots \zeta_{n_r}$,
where $1 \le n_i \le p$ and  $r \ge 1$ are integers. We shall show that $\ms{T}^\p(P)=\ms{T}(P, \psi(\ti{C}))$ by proving that $\mc{F}_P(\ms{T}^\p(P))$ is invariant under intersection with $\xi$. The argument of [Be04,  Lem.4.2] carries over verbatim to the present case and shows that $\mc{F}_P(\ms{T}^\p(P))$ is spanned by the classes $({k_1}_\ast [\psi(\ti{C})]) \ast \cdots \ast ({k_r}_\ast [\psi(\ti{C})])$ for all sequences $k_1, \ldots, k_r$ of positive integers. Therefore, it is enough to show that $\xi \cdot \big(({k_1}_\ast [\psi(\ti{C})]) \ast \cdots \ast ({k_r}_\ast [\psi(\ti{C})])\big)$ belongs to $\mc{F}_P(\ms{T}^\p(P))$. Consider the composition
$$v\colon \xymatrix{ \ti{C}^r \ar[r]^-{\vec{\psi}} & P^r \ar[r]^-{\vec{k}} & P^r \ar[r]^-m & P},$$
where $\vec{\psi} = (\psi, \ldots, \psi)$, $\vec{k}=(k_1,\ldots, k_r)$ with $k_i\colon x\mapsto k_ix$, and $m\colon P^r\ra P$ is the addition morphism. It suffices to show that $v_*v^*\xi \in \mc{F}_P(\ms{T}^\p(P))$.

Let $\Theta_{\ti{J}}$ and $\Xi$ be theta divisors on $\ti{J}$ and $P$, respectively, such that $\Theta_{\ti{J}}$ restricts to $2\Xi$ on $P$. Let $\ms{L}_P$ be the Poincar\'e bundle on $P \times P$  normalized by the conditions $\ms{L}_P{|_{\{b\}\times P}}\simeq \ms{L}_P{|_{P \times \{b\}}} \simeq \mc{O}_P(\Xi - \Xi_b)$ and let $q_i, q_{ij}$ be the projections of $\ti{C}^r$ onto the $i^{\mrm{th}}$, $i^{\mrm{th}}$ and $j^{\mrm{th}}$ factors, respectively. The same argument as in the proof of [Be04, Prop.4.1] shows that
\begin{equation}\label{vastxi}
v^\ast \xi = \sum_i k_i^2q_i^\ast \psi^\ast \xi -\sum_{i<j}k_ik_j q_{ij}^\ast (\psi, \psi)^\ast c_1(\ms{L}_P).
\end{equation}

Next, we compute $(\psi, \psi)^\ast c_1(\ms{L}_P)$. Let $\Theta_J$ be a theta divisor on $J$. Also, let $\ms{L}_{\ti{J}}$ and $\ms{L}_J$ be the Poincar\'e bundles on $\ti{J} \times \ti{J}$ and $J\times J$, respectively, normalized by the conditions analogous to those for $\ms{L}_P$. The result [BL, Prop.12.3.4, p.374] in our notation states that $2\Theta_{\ti{J}} \sim_{\mrm{alg}} \mrm{Nm}^*(\Theta_J) + u^*\Xi$. Consequently, the following identity holds in $A(\ti{J}\times \ti{J})$ 
\begin{equation} \label{pPPtoJJ}
(u, u)^*c_1(\ms{L}_P) = 2 c_1(\ms{L}_{\ti{J}})- (\mrm{Nm},\mrm{Nm})^*c_1(\ms{L}_J).
\end{equation}
Let $\Delta_{\ti{C}}$ and $\Delta_{C}$ be the diagonals in $\ti{C} \times \ti{C}$ and $C\times C$, respectively. Applying the See-saw theorem [Mu70, p.54], we see that 
\begin{equation} \label{LJtoBB}
(\ti{\varphi},\ti{\varphi})^\ast \ms{L}_{\ti{J}} \simeq \mc{O}_{\ti{C}^2}(\Delta_{\ti{C}} -\ti{C} \times \ti{o} -\ti{o} \times \ti{C})
\end{equation}
and likewise $(\varphi,\varphi)^\ast \ms{L}_{J} \simeq \mc{O}_{C^2}(\Delta_{C} -C \times o -o \times C). $  Using the commutative diagram 
$$\xymatrix{
\ti{C} \times \ti{C} \ar[d]_-{(\pi, \pi)} \ar[r]^-{( \ti{\varphi}, \ti{\varphi} )}  & \ti{J} \times \ti{J} \ar[d]^{ (\mrm{Nm}, \mrm{Nm}) } \\
C\times C \ar[r]^-{(\varphi, \varphi)}& J \times J,
}$$
we obtain 
\begin{equation} \label{pPNtoNN}
(\mrm{Nm} \circ \ti{\varphi}, \mrm{Nm} \circ \ti{\varphi})^\ast \ms{L}_{J} \simeq (\pi, \pi)^\ast \mc{O}_{C^2}(\Delta_C -C\times o - o\times C).
\end{equation}
Since $\psi = u\circ \ti{\varphi}$, then from (\ref{pPPtoJJ}), (\ref{LJtoBB}) and (\ref{pPNtoNN}) we get
$$(\psi, \psi)^\ast c_1(\ms{L}_P) \sim_{\mrm{alg}} 2(\Delta_{\ti{C}} - \ti{C} \times \ti{o} - \ti{o} \times \ti{C}) - (\pi, \pi)^\ast (\Delta_C - C\times o - o\times C).$$
Furthermore, $(\pi, \pi)^\ast \Delta_C = \Delta_{\ti{C}}+(1, \iota)^\ast \Delta_{\ti{C}},$
where $1\colon \ti{C} \ra \ti{C}$ denotes the identity morphism, and therefore, 
\begin{equation} \label{rrLP}
(\psi, \psi)^\ast c_1(\ms{L}_P) \sim_{\mrm{alg}} \Delta_{\ti{C}}-(1,\iota)^*\Delta_{\ti{C}}.
\end{equation}
Substituting (\ref{rrLP}) into the identity (\ref{vastxi}), we see that $v^\ast \xi$ is algebraically equivalent to a linear combination of divisors of the form $q_i^\ast \ti{o}$ and $q_{ij}^\ast \beta^\ast \Delta_{\ti{C}}$, where $\beta$ is one of the morphisms $(1, 1)$ or $(1, \iota)$. 
The cycles $v_\ast q_{ij}^\ast \Delta_{\ti{C}}$ and $v_\ast q_{ij}^\ast(1,\iota)^\ast \Delta_{\ti{C}}$ are proportional to a cycle of the form $(l_{1\ast} [\psi(\ti{C})])\ast \cdots \ast (l_{(r-1)\ast} [\psi(\ti{C})])$, where $(l_1, \ldots, l_{r-1})$ is $$(k_1, \ldots, \widehat{k_i}, \ldots, \widehat{k_j}, \ldots, k_r, k_i+k_j) \textrm{ and } (k_1, \ldots, \widehat{k_i}, \ldots, \widehat{k_j}, \ldots, k_{r}, k_i-k_j),$$ respectively, and the symbol $\widehat{k_j}$ means that $k_j$ is omitted from the list. Since $v_\ast q_i^\ast \ti{o}$ is proportional to the cycle  $(k_{1\ast} [\psi(\ti{C})])\ast \cdots \ast (k_{r\ast} [\psi(\ti{C})])$ with $k_{i\ast} [\psi(\ti{C})]$ omitted, we see that $v_\ast v^\ast \xi$ belongs to $\mc{F}_P(\ms{T}^\p(P))$. 
\end{proof}

\begin{RemNotPolynInXi}\label{RemNotPolynInXi}\emph{
If $P$ is generic and of dimension $p\ge 5$, then by the proof of [Fa, Thm.4.5, p.117], the class $([\psi(\ti{C})]^{*r})_{(2)}$ is non-zero in $A(P)$ for $1 \le r \le p-3$. Therefore, by Fourier duality, $\zeta_1^j\zeta_3 \ne 0$ in $A(P)$ for $0 \le j \le p - 4$.}
\end{RemNotPolynInXi}

\begin{RemUniquenessOfT}\emph{The Torelli theorem implies that every principally polarized Jacobian has a unique tautological ring in the sense of [Be04]. It is well known that the analog of the Torelli theorem does not hold for Prym varieties, [Do, IL, Ve].  We do not know whether the tautological ring of the pair $(P, \psi(\ti{C}))$ is always independent of the choice of an Abel-Prym curve. However, we know that $\ms{T}(P, \psi(\ti{C}))$ is independent of this choice when $P$ is a general Prym of dimension $2 \le p \le 4$. This follows from the connectedness of the general fiber of the Prym map, see [Do, \S 6; Iz95]. Furthermore, due to the generic injectivity of the Prym map for curves of genus $g \ge 7$, see [De, Ka, FS, We87], the tautological ring $\ms{T}(P, \psi(\ti{C}))$ of a generic Prym variety $P$ of dimension $\ge 6$ depends only on $P$.}
\end{RemUniquenessOfT}

\section{Algebraic equivalence of special subvarieties.} \label{SecOnAlgEq}
In this section we let $r,d,g$ be integers such that $0<2r<d<2g$.  Also, $S$ will denote a smooth connected but not necessarily complete curve. For a morphism $X\ra S$, the fiber over $s$ is denoted by $X_s$, and for a sheaf $\mc{F}$ on $X$, $\mc{F}_s:=\mc{F}_{|X_s}$ is the restriction. If $X$ is an integral projective scheme, we let $\mrm{Pic}^0_X$ be the connected component of the identity in the Picard scheme of $X$. If $X$ is a smooth curve, we let $X_d$ denote the $d^{\mrm{th}}$ symmetric product of $X$. An integral curve with $n$ ordinary double points (resp., $n$ ordinary cusps) and no other singularities will be called $n$-nodal (resp., $n$-cuspidal). 

$B$ will denote a smooth connected curve of genus $g-1$. $B_{pq}$ will denote the $1$-nodal curve obtained by gluing distinct points $p$ and $q$ on $B$.  Also, $B_{pp}$ will denote the $1$-cuspidal curve, whose normalization is $B$, and such that the point $p\in B$ maps to the cusp of $B_{pp}$. For simplicity, both normalization morphisms $B \ra B_{pq}$ and $B \ra B_{pp}$ will be denoted by $\nu$ and the distinction will be clear from the context.

\subsection{Special subvarieties} \label{Special subvarieties} Assume that $C$ has a complete $g^r_d$, which is viewed as a subvariety $G_d \subset C_d$ isomorphic to $\bP^r$. Let us recall the definition of the  special subvarieties $V_0$ and  $V_1$ of $P$ associated to $G_d$, cf.\! [Be82]. First, we assume that $G_d$ contains a reduced divisor.  Consider the following commutative diagram 
$$\xymatrix{
\ti{C}_d \ar[r]^-{\ti{\varphi}_d} \ar[d]_-{\pi_d} & \ti{J} \ar[d]^-{\mrm{Nm}} \\
C_d \ar[r]^-{\varphi_d} & J,
}$$ 
where the horizontal maps are abelian sum mappings $\varphi_d\colon D \mapsto \mc{O}_C(D-do)$, $\ti{\varphi}_d\colon \ti{D} \mapsto \mc{O}_{\ti{C}}(\ti{D}-d\ti{o})$ and $\pi_d$ is induced by $\pi\colon \ti{C} \ra C$.
The variety $\ti{\varphi}_d(\pi_d^{-1}(G_d))$ is in the kernel of the norm map $\mrm{Nm}\colon \ti{J} \ra J$ and has two connected components $V_0$ and $V_1$, [Be82, p.365; We81, p.98]. After a translation of one of the components $V_0$ or $V_1$, we assume that both of them are contained in $P$ and, by definition, $V_0$ and $V_1$ are called \emph{special subvarieties} of $P$. The union $V_0 \cup V_1$ is denoted by $V$. In the case when $G_d$ has a base-point of multiplicity $\ge 2$, $V$ is non-reduced and we may define it as a cycle $V:=\ti{\varphi}_{d*}(\pi_d^*(G_d))$ with multiplicities, [Be82, p.359]. In this case $V_0$ and $V_1$ can also be defined as cycles.

By Clifford's theorem, the inequalities $0 < d < 2g$ imply that $d>2r$, except when the $g^r_d$ is the canonical system or a multiple of a $g^1_2$. In these two cases the special subvarieties are not even homologically equivalent, see [Be82, Rem.3, p.362 and p.366]. However, if $2r< d$, then the subvarieties $V_0$ and $V_1$ are homologically equivalent, i.e., $V_0$ and $V_1$ have the same cohomology class 
$$2^{d-2r-1}\cdot\frac{\xi^{g-r-1}}{(g-r-1)!}$$
in $H^{2(g-r-1)}(P, \mb{Z})$, see [Be82, Prop.1, p.360 and Thm.1, p.364].

\subsection{Compactified Jacobians and the Abel map}\label{Compactified Jacobians} We shall use the results of [EGK] in the sequel and we recall the necessary notation. Let $\mc{C} \ra S$ be a flat family of integral curves. We assume that the family $\mc{C}/S$ has a section $\sigma$ whose image lies in the smooth locus of the morphism $\mc{C} \ra S$. Given an integer $n$, a \emph{torsion-free rank one sheaf of degree $n$} on $\mc{C}/S$ is an $S$-flat coherent $\mc{O}_{\mc{C}}$-module $\mc{F}$ such that  $\mc{F}_s$ is a torsion free rank one sheaf on the fiber $\mc{C}_s$ and $\chi(\mc{F}_s)-\chi(\mc{O}_{\mc{C}_s}) =n$ for every $s \in S$. There is a projective $S$-scheme $\bar{J}_{\mc{C}/S}^{n}$, called the \emph{compactified Jacobian of $\mc{C}/S$}, that parametrizes torsion-free rank one sheaves of degree $n$ on the fibers of $\mc{C}/S$, see [EGK, p.594; Al]. There is also an open subscheme $J_{\mc{C}/S}^{n} \subset \bar{J}_{\mc{C}/S}^{n}$, called the (\emph{generalized}) \emph{Jacobian of $\mc{C}/S$}, parametrizing those sheaves that are invertible. For a single integral curve $X$, we use the notation $J^n_X$ and $\bar{J}^n_X$, for the analogous notions. The section $\sigma\colon S \ra \mc{C}$ gives an invertible sheaf $\mc{N}$ of degree one on $\mc{C}$, which determines an \emph{Abel map} $A_{\mc{N}}\colon \mc{C} \ra \bar{J}^0_{\mc{C}/S}$, see [EGK, p.595].  On the fiber $\mc{C}_s$, the Abel map is given by  $x \mapsto \frak{m}_x \otimes \mc{N}_s$, where $\frak{m}_x$ is the maximal ideal in the local ring $\mc{O}_{\mc{C}_s,x}$. If the geometric fibers of $\mc{C}\ra S$ have double points at worst, then by [EGK, Thm.2.1] the Abel map induces an isomorphism of group-schemes:
$$A_{\mc{N}}^*\colon \mrm{Pic}^0_{\bar{J}^0_{\mc{C}/S}} \ra J^0_{\mc{C}/S}.$$

\subsection{Presentation schemes} \label{Presentation schemes}
In Section \ref{The Wirtinger family} we shall use the construction of [AK], called the \emph{presentation scheme}, which we recall next. Let $X$ be an integral curve with a unique double point and let $\nu\colon X^\p \ra X$ be the normalization. Given an integer $n$, the presentation scheme $P_X^n$ parametrizes injective morphisms $h\colon L \hookrightarrow \nu_* M$, called presentations, such that $L \in \bar{J}^n_X$ and $M \in J^n_{X^\p}$. 
By [AK], the presentation scheme $P_X^n$ fits into a diagram
$$\xymatrix{ 
P_X^n \ar[r]^-{\kappa} \ar[d]_-\lambda & \bar{J}^n_X \\
J^n_{X^\p},
}$$
where $\kappa$ and $\lambda$  send ($h\colon L \hookrightarrow \nu_* M)$ to $L$ and $M$, respectively. Note that in [AK] the degree of $L$ is $\chi(L)$, not $\chi(L) - \chi(\mc{O}_X)$, as it is for us.  

When $X=B_{pq}$ is $1$-nodal, $\lambda\colon P_{B_{pq}}^n \ra J^n_{B}$ is a $\bP^1$-bundle, which has two distinguished sections $s_p$ and $s_q$. The morphism $\kappa\colon P_{B_{pq}}^n \ra  \bar{J}^n_{B_{pq}}$ identifies the images $\mrm{Im}(s_p)$ and $\mrm{Im}(s_q)$ with a shift and is an isomorphism outside of $\mrm{Im}(s_p) \cup \mrm{Im}(s_q)$. More precisely, if $M \in J^n_{B}$, then $\kappa$ identifies $s_p(M)$ with $s_q(M\otimes \mc{O}_{B}(q-p))$. Furthermore, the common image of $\kappa \circ s_p$ and $\kappa\circ s_q$ coincides with $\pl \bar{J}^n_{B_{pq}}$, the locus of non-invertible sheaves. 
 
When $X=B_{pp}$ is $1$-cuspidal, let $p^\p \subset B$ be the fiber over the cusp of $B_{pp}$. There is an embedding $s_{p^\p}\colon J^n_{B}\times p^\p \ra P^n_{B_{pp}}$ such that $\lambda \circ s_{p^\p} \colon J^n_{B}\times p^\p \ra J^n_{B}$ is the first projection. The scheme $P^n_{B_{pp}}$ is non-reduced along $\mrm{Im}(s_{p^\p})$, although $\bar{J}^n_{B_{pp}}$ itself is reduced. The morphism $\kappa$ is bijective but is not an isomorphism: it maps the non-reduced locus $\mrm{Im}(s_{p^\p})$ to the reduced subscheme $\pl \bar{J}^n_{B_{pp}} \subset \bar{J}^n_{B_{pp}}$ and is an isomorphism outside of $\mrm{Im}(s_{p^\p})$.   
 
\subsection{Families of Brill-Noether loci} \label{Brill-Noether loci}  The scheme $\mc{C}\times_S \bar{J}^{d}_{\mc{C}/S}$ admits a Poincar\'e sheaf, which represents the universal family, see [EGK, p.594] and references therein. Using this Poincar\'e sheaf, the construction in [AC, p.6] carries over verbatim to the present situation and gives an $S$-subscheme $\mc{W}^r_d(\mc{C})$ of $\bar{J}^{d}_{\mc{C}/S}$, whose fiber over $s$ is $\ol{W}^r_d(\mc{C}_s)$, the subscheme of $\bar{J}^{d}_{\mc{C}_s}$ parametrizing sheaves $F$ with $h^0(\mc{C}_s, F) \ge r+1$. The sets of locally free and non-locally free elements of $\ol{W}^r_d(\mc{C}_s)$ are denoted by $W^r_d(\mc{C}_s)$ and $\pl \ol{W}^r_d(\mc{C}_s)$, respectively. If $\mc{C}_s$ is smooth then $\ol{W}^r_d(\mc{C}_s)=W^r_d(\mc{C}_s)$ is the classical Brill-Noether scheme as in [ACGH, p.153]. 

\medskip

In Section \ref{The Wirtinger family} we shall make repeated use of the following two observations. 

\begin{flatness criterion}[Flatness criterion] \label{flatness criterion} If all fibers of $\mc{W}^r_d(\mc{C}) \ra S$ are reduced and have the expected dimension $\rho(g,r,d)$, then $\mc{W}^r_d(\mc{C})$ is flat over $S$.
\end{flatness criterion} 
\begin{proof} Since $S$ is a smooth connected curve by our ongoing assumption, then by [Ha, Prop.9.7, p.257] it suffices to check that no component of $\mc{W}^r_d(\mc{C})$ is contained in a fiber $\ol{W}^r_d(\mc{C}_s)$ for some $s \in S$. Since $\mc{W}^r_d(\mc{C})$ is determinantal, then each irreducible component of $\mc{W}^r_d(\mc{C})$ has dimension $\ge \rho(g,r,d)+1$ at every point. But $\dim \ol{W}^r_d(\mc{C}_s)=\rho(g,r,d)$ for all $s\in S$ by assumption, and therefore, no component of $\mc{W}^r_d$ can be contained in a fiber (this argument is adapted from [HM, p.267]).
\end{proof}

\begin{principle of connectedness}[Principle of connectedness] \label{principle of connectedness}
Let $X \ra S$ be a flat family of projective schemes. If there exists a point $s_0 \in S$ such that $X_{s_0}$ is connected and reduced, then $X_s$ is connected for each point $s\in S$.  
\end{principle of connectedness}
\begin{proof}
The hypotheses of the proposition imply that $h^0(\mc{O}_{X_{s_0}})=1$. By upper semicontinuity of $s \mapsto h^0(\mc{O}_{X_{s}})$, there is a Zariski open subset $U \subset S$ such that for all points $s \in U$, $X_s$ is connected. By [Ha, Ex.11.4, p.281], $X_s$ is connected for each point $s \in S$. 
\end{proof}

\subsection{Main result} \label{The Wirtinger family}
The following proposition is due to E. Izadi. 

\begin{Dr. Izadi's argument} \label{Dr. Izadi's argument}
Let $C$ be a smooth curve of genus $g$, which has a complete $g^r_d$ with $0<2r<d <2g$, and let $\ti{C} \ra C$ be a connected \'etale double cover. If $W^r_d(C)$ is connected and either $W^r_{d-1}(C)$ or $W^{r+1}_{d+1}(C)$ is nonempty, then the special subvarieties $V_0$ and $V_1$ associated to the $g^r_d$ are algebraically equivalent.
\end{Dr. Izadi's argument}
\begin{proof}
The cover $\ti{C} \ra C$ determines a point of order two in $J^0_C$. Translating by $\mc{O}_C(do)$ and going to the dual abelian variety of $J^d_C$, gives a point of order two in $\mrm{Pic}^0(J^d_C)$, which determines an \'etale double cover $\wti{J}^d_C \ra J^d_C$. Assume that $W^r_{d-1}(C)$ is nonempty. Take $L \in W^r_{d-1}(C)$ and let $C \rightarrow W^r_d(C)$ be the embedding $x \mapsto L(x)$. 
We may check that there is a commutative diagram
$$\xymatrix{
\ti{C} \ar[r] \ar[d] & \wti{W}^r_d(C) \ar[r] \ar[d] & \wti{J}^d_C \ar[d] &\\
C \ar[r] & W^r_d(C) \ar[r] & J^d_C,
}$$
where the two squares are Cartesian.  The double cover $\wti{J}^d_C \ra J^d_C$ restricts to the connected double cover $\ti{C} \ra C$, and therefore, the intermediate double cover $\wti{W}^r_d(C)\ra W^r_d(C)$ is non-trivial. Since $W^r_d(C)$ is connected, this shows that $\wti{W}^r_d(C)$ is also connected. The special subvarieties $V_0$ and $V_1$ are members of a family of cycles over $\wti{W}^r_d(C)$. Since $\wti{W}^r_d(C)$ is connected, then $V_0$ and $V_1$ are algebraically equivalent. The case when $W^{r+1}_{d+1}(C)$ is nonempty is analogous (embed $C$ in $W^r_d(C)$ by $x \mapsto L(-x)$ for a fixed $L \in W^{r+1}_{d+1}(C)$). 
\end{proof}

If $\rho(g,r,d) < \min\{r+1, g-d+r\}$ and $C$ is general in the sense of Brill-Noether theory, both  $W^r_{d-1}(C)$ and $W^{r+1}_{d+1}(C)$ are empty, because $\rho(g,r,d-1)$ and $\rho(g,r+1,d+1)$ are negative, [ACGH, Thm.1.5, p.214]. Nevertheless, in what follows we shall show that $V_0$ and $V_1$ are algebraically equivalent whenever $\rho(g,r,d)>0$ and $W^r_d(C)$ is reduced and of the expected dimension. 

\medskip

The \emph{Wirtinger cover} is the \'etale double cover $\ti{B}_{pq} \ra B_{pq}$, where the curve $\ti{B}_{pq}$ is obtained from two copies of $B$ by identifying $p$ and $q$ on one copy with $q$ and $p$, respectively, on the other copy. Using [Be77, 6.1] we may find a smooth connected curve $S$ such that $\ti{C} \ra C$ and $\ti{B}_{pq} \ra B_{pq}$ vary in a family of double covers 
$$\xymatrix{
\ti{\mc{C}} \ar[rr]^-{\tau} \ar[dr] & & \mc{C} \ar[dl] \\
& S &
}$$
($S$ in [Be77] is not the same as our $S$) with the following properties: the fibers of $\mc{C}\ra S$ are integral and at worst nodal, $\ti{\mc{C}}$ and $\mc{C}$ are flat over $S$, $\tau$ is \'etale, and there is a section $\sigma$ of $\mc{C} \ra S$, which induces a degree one invertible sheaf $\mc{N}$ on $\mc{C}$, as in Section \ref{Compactified Jacobians}. Assume that $\ti{B}_{pq} \ra B_{pq}$ lies over the special point $s_0 \in S$. By Section \ref{Brill-Noether loci} there is an associated family $\mc{W}^r_d(\mc{C})\ra S$.  

\medskip

The \'etale double cover $\tau\colon\ti{\mc{C}} \ra \mc{C}$ induces a morphism $\mc{O}_{\mc{C}} \hookrightarrow \tau_\ast \mc{O}_{\ti{\mc{C}}}$, whose cokernel is an invertible sheaf $\mc{L}$ with the property $\mc{L}^{2}\simeq \mc{O}_{\mc{C}}$. Let $\mc{M}:=(A_{\mc{N}}^*)^{-1}(\mc{L})$, then $\mc{M}^{2}$  is isomorphic to the pull-back of an invertible sheaf on $S$. After replacing $S$ by a Zariski open subset containing $s_0$, we may assume that $\mc{M}^{2}$ is trivial. The relative spectrum $\wti{J}^{0}_{\mc{C}/S}:= \mrm{\textbf{Spec}}(\mc{O}_{ \bar{J}^{0}_{\mc{C}/S}}\oplus \mc{M})$ over $\bar{J}^{0}_{\mc{C}/S}$ gives an \'etale double cover $\wti{J}^{0}_{\mc{C}/S} \ra \bar{J}^{0}_{\mc{C}/S}$. Pulling back via the isomorphism $\bar{J}^{d}_{\mc{C}/S}\ra \bar{J}^{0}_{\mc{C}/S}$ of tensoring with $\mc{N}^{-d}$, we obtain an \'etale double cover $\wti{J}^{d}_{\mc{C}/S} \ra \bar{J}^{d}_{\mc{C}/S}$, whose restriction to $\mc{W}^r_d(\mc{C})$ is denoted by  
$\wti{\mc{W}}^r_d(\mc{C}) \ra \mc{W}^r_d(\mc{C})$. Over the special point $s_0\in S$, we use the notation $\wti{J}^d_{B_{pq}} \ra \bar{J}^d_{B_{pq}}$ and $\wti{W}^r_d(B_{pq}) \ra \ol{W}^r_d(B_{pq})$ for the induced double covers. We have:

\begin{wirtcon} \label{wirtcon}
If $0 < \rho(g,r,d) < \min\{r+1,g-d+r\}$, then a general 1-nodal curve $B_{pq}$ of arithmetic genus $g$ satisfies: (1)  $\ol{W}^r_d(B_{pq})$ is connected, reduced, and has the expected dimension $\rho(g,r,d)$; (2) $\wti{W}^r_d(B_{pq})$ is connected. 
\end{wirtcon}

Assuming the above proposition, we prove the main result of this section:

\begin{algeq} \label{algeq}
Let $C$ be a smooth curve of genus $g$, which has a complete $g^r_d$, and such that $W^r_d(C)$ is reduced and of dimension $\rho(g,r,d)$. Let $\ti{C} \ra C$ be a connected \'etale double cover and let $V_0,V_1$ be the special subvarieties associated to the $g^r_d$. If $\rho(g,r,d)>0$, then $V_0$ and $V_1$ are algebraically equivalent.
\end{algeq}
\begin{proof} The idea of the proof is taken from [IT]. If $\rho(g,r,d) \ge \min\{r+1, g-d+r\}$, then either $W^r_{d-1}(C)$ or $W^{r+1}_{d+1}(C)$ is nonempty [Ke, KL], and Proposition \ref{Dr. Izadi's argument} applies. Thus, we may assume that $0 < \rho(g,r,d) < \min\{r+1, g-d+r\}$. 

Let $B_{pq}$ be a 1-nodal curve satisfying the conclusions of Proposition \ref{wirtcon}.  We shall degenerate $\wti{W}^r_d(C)$ to $\wti{W}^r_d(B_{pq})$ and apply the principle of connectedness from Section \ref{Brill-Noether loci}. Consider the flat family $\tau\colon \ti{\mc{C}} \ra \mc{C}$ over $S$ and the associated family $\wti{\mc{W}}^r_d(\mc{C}) \ra \mc{W}^r_d(\mc{C})$, as described before Proposition \ref{wirtcon}. The schemes $W^r_d(C)$ and $\ol{W}^r_d(B_{pq})$ are reduced, have dimension $\rho(g,r,d)$, and appear as fibers of $w_{\mc{C}}\colon \mc{W}^r_d(\mc{C}) \ra S$. After shrinking $S$, if necessary, we may assume that all fibers of $w_{\mc{C}}$ are reduced and have dimension $\rho(g,r,d)$, hence $w_{\mc{C}}$ is flat by Proposition \ref{flatness criterion}. This implies that the morphism $\wti{\mc{W}}^r_d(\mc{C}) \ra S$ is also flat and all of its fibers are reduced. Since  $\wti{W}^r_d(B_{pq})$ is connected, then by Proposition \ref{principle of connectedness}, so is $\wti{W}^r_d(C)$. As in the proof of Proposition \ref{Dr. Izadi's argument}, this implies that $V_0$ and $V_1$ are algebraically equivalent.
\end{proof}

Before giving the proof of Proposition \ref{wirtcon}, we need two preliminary observations. 
First, let us show that $\wti{J}_{B_{pq}}^d \ra \bar{J}^d_{B_{pq}}$ has a description analogous to the Wirtinger cover $\ti{B}_{pq} \ra B_{pq}$. Using the notation from Section \ref{Presentation schemes}, we see that the morphism $\kappa\colon P^d_{B_{pq}} \ra \bar{J}^d_{B_{pq}}$ is the normalization. Let $\alpha\colon B_{pq} \ra  \bar{J}^d_{B_{pq}}$ be the composition of the Abel map (Section \ref{Compactified Jacobians})  followed by the morphism $\bar{J}^0_{B_{pq}} \ra\bar{J}^d_{B_{pq}}$ of tensoring by $\mc{N}_{s_0}^d$. By the universal property of normalization, the composition $\alpha \circ \nu\colon B\ra \bar{J}^d_{B_{pq}}$ induces a morphism $\ti{\alpha}\colon B \ra P_{B_{pq}}^d$, such that the diagram 
$$\xymatrix{
B \ar[r]^-{\ti{\alpha}} \ar[d]_-\nu & P_{B_{pq}}^d \ar[d]^-\kappa &\\
B_{pq} \ar[r]_-\alpha & \bar{J}^d_{B_{pq}}
}$$
commutes. Applying $\mrm{Pic}^0$, we get a commutative diagram:
$$\xymatrix{
\mrm{Pic}^0_B     & \ar[l]_-{\ti{\alpha}^*}  \mrm{Pic}^0_{P_{B_{pq}}^d} \\
\mrm{Pic}^0_{B_{pq}} \ar[u]^-{\nu^*} & \ar[l]^-{\alpha^*} \ar[u]_-{\kappa^*} \mrm{Pic}^0_{\bar{J}^d_{B_{pq}}}.
}$$
The composition $\lambda \circ \ti{\alpha}\colon B \ra P_{B_{pq}}^d \ra J^d_B$ is an Abel map (see Section \ref{Presentation schemes} for the definition of $\lambda$),  which implies that $\ti{\alpha}^*$ is injective. By construction preceding Proposition \ref{wirtcon}, the invertible sheaves $\mc{L}_{s_0}$ and $\mc{M}_{s_0}$ that determine $\ti{B}_{pq} \ra B_{pq}$ and  $\wti{J}_{B_{pq}}^d \ra \bar{J}^d_{B_{pq}}$, respectively, satisfy $\alpha^*(\mc{M}_{s_0})=\mc{L}_{s_0}$. Also, $\mc{L}_{s_0}$ induces the trivial double cover of $B$, i.e., $\nu^*(\mc{L}_{s_0})=\mc{O}_B$. Therefore, using the commutativity of the above diagram and the injectivity of $\ti{\alpha}^*$, we see that $\kappa^*(\mc{M}_{s_0}) = \mc{O}_{P_{B_{pq}}^d}$. This shows that there is a Cartesian square
$$\xymatrix{P_{B_{pq}}^d \coprod P_{B_{pq}}^d \ar[r] \ar[d] & \wti{J}^d_{B_{pq}} \ar[d] & \\
P_{B_{pq}}^d \ar[r]^-\kappa & \bar{J}^d_{B_{pq}},}$$  
where the left vertical arrow is the trivial double cover. We conclude that $\wti{J}^d_{B_{pq}}$ is obtained from $P_{B_{pq}}^d \coprod P_{B_{pq}}^d$ by gluing $\mrm{Im}(s_p)$ and $\mrm{Im}(s_q)$ on one copy of $P_{B_{pq}}^d$ to $\mrm{Im}(s_q)$ and $\mrm{Im}(s_p)$, respectively, on the other copy (the gluing is with a shift, as in the case of $\bar{J}^d_{B_{pq}}$ described in Section \ref{Presentation schemes}). 

\medskip 

Next, consider the induced double cover $\wti{W}^r_d(B_{pq}) \ra \ol{W}^r_d(B_{pq})$ and define 
$$
W(p,q) := \kappa^{-1}\big(\ol{W}^r_d(B_{pq})\big) \qquad
W(p):=\mrm{Im}(s_p)\cap W(p,q) \qquad W(q):=\mrm{Im}(s_q)\cap W(p,q).
$$
From the above description of $\wti{J}^d_{B_{pq}}$ we see that $\wti{W}^r_d(B_{pq})$ is obtained from two copies of $W(p,q)$ by gluing $W(p)$ and $W(q)$ on one copy to $W(q)$ and $W(p)$, respectively, on the other copy. Therefore, to show that $\wti{W}^r_d(B_{pq})$ is connected, it suffices to prove that $W(p,q)$ is connected and $W(p)$, $W(q)$ are nonempty. Non-emptiness of $W(p)$ and $W(q)$ can be seen as follows. Since  $\dim W^r_{d-1}(B)\ge \rho(g-1,r,d-1)=\rho(g,r,d)-1\ge 0$ (the last inequality holds by our ongoing assumption $\rho(g,r,d) >0$) and $\pl \ol{W}^r_d(B_{pq})=\{\nu_*M \,|\, M\in W^r_{d-1}(B)\}$, then $\pl \ol{W}^r_d(B_{pq})$ is nonempty, [Ke, KL]. It follows from Section \ref{Presentation schemes} that $\kappa$ maps each of $W(p)$ and $W(q)$ onto $\pl \ol{W}^r_d(B_{pq})$, which implies that $W(p)$ and $W(q)$ are nonempty. Connectedness of $W(p,q)$ (for a general $1$-nodal curve $B_{pq}$) is much harder to show and is the main point of the proof of Proposition \ref{wirtcon} below. The idea is to let $p$ and $q$ come together and to consider the analogous locus $W(p,p)$ for the 1-cuspidal curve $B_{pp}$. Although the scheme $W(p,p)$ turns out to be connected, it is also non-reduced. Hence, we may have $h^0(\mc{O}_{W(p,p)}) > 1$, and therefore, we cannot apply the principle of connectedness  directly to conclude that $W(p,q)$ is connected. To overcome this difficulty, we shall use the determinantal loci $Y(x,y)$, which are introduced in the next paragraph.  

\medskip

Let us introduce two last bits of notation, state two lemmas (whose proofs are  at the end of the section), and give the proof of Proposition \ref{wirtcon}. First,  by [EH] if $X$ is an integral curve with double points at worst, there is a scheme $\ol{G}^r_d(X)$ parametrizing pairs $(L, V)$ such that $L \in \ol{W}^r_d(X)$ and $V \subset H^0(X, L)$ is a subspace of dimension $r+1$. There is a forgetful morphism $\ol{G}^r_d(X) \ra \ol{W}^r_d(X)$ and a subscheme $G^r_d(X)\subset \ol{G}^r_d(X)$ pararmetrizing pairs $(L, V)$ with $L$ locally free. Second, given points $x,y \in B$ (not necessarily distinct), define  $$Y(x,y):=\{M \in J^d_B \,|\, h^0(M) \ge r+1 \text{ and } h^0(M(-x-y)) \ge r\}$$ as a subscheme of $J^d_B$ with its natural structure of a determinantal locus, whose expected dimension can be computed using [Fu, Ch.14.3, p.249], and is equal to $\rho(g,r,d)$, see [Ar, Appendix to Ch.4] for details.

\begin{wirtcon helper 1} \label{wirtcon helper 1}
If $\rho(g,r,d) < \min\{r+1, g+r-d\}$, then for a general $1$-cuspidal curve $B_{pp}$ of arithmetic genus $g$, the scheme $Y(p,p)$ is connected, reduced, and of dimension $\rho(g,r,d)$.
\end{wirtcon helper 1}

\begin{wirtcon helper 2} \label{wirtcon helper 2} If both $W^{r+1}_d(B)$ and $W^r_{d-2}(B)$ are empty, then for any distinct points $p,q \in B$, the underlying topological spaces of the schemes $Y(p,q)$ and $W(p,q)$ are homeomorphic.
\end{wirtcon helper 2}

\begin{proof}[Proof of Prop.\ref{wirtcon}]  
Say that an integral curve $X$ of arithmetic genus $g$ has property $\mc{P}$, if $\ol{W}^r_d(X)$ is connected, reduced, and of dimension $\rho(g,r,d)$. 
Note that property $\mc{P}$ is open and depends on the integers $r,d$.  By [EH, Thm.4.5 and Sec.9], there exists a rational $g$-cuspidal curve $X$ for which $\ol{G}^r_d(X)$ is connected, reduced, and of the expected dimension, which implies that $\mc{P}$ holds for $X$. It follows that $\mc{P}$ holds for a general 1-nodal curve. This proves part (1) of the proposition. 

Connectedness of $\wti{W}^r_d$ is an open property of $1$-nodal curves. Therefore, to prove part (2) of the proposition, it suffices to exhibit a single $1$-nodal curve $B_{pq}$ such that $\wti{W}^r_d(B_{pq})$ is connected. This will be done with the aid of $1$-cuspidal curves. By Lemma \ref{wirtcon helper 1}, there exists a $1$-cuspidal curve $B_{pp}$ such that $Y(p,p)$ is connected, reduced, and of dimension $\rho(g,r,d)$. Moreover, by the proof of Lemma \ref{wirtcon helper 1}, we may assume that both $W^{r+1}_d(B)$ and $W^r_{d-2}(B)$ are empty. Let us show that there exists a point $q \in B \backslash \{p\}$ such that $Y(p,q)$ is connected. Take a smooth connected but not necessarily complete curve $T$  parametrizing divisors $\{p+q_t \,|\, t\in T\}$ on $B$ such that $q_t \ne p$ for all $t\in T \backslash \{t_0\}$ and $q_{t_0} = p$.  Consider a family $\ms{Y} \ra T$, such that the fiber over a point $t \in T$ is $Y(p,q_t)$. Since $Y(p,p)$ is connected, reduced, and of the expected dimension $\rho(g,r,d)$, we may replace $T$ with a Zariski open neighborhood of $t_0$, if necessary, to ensure that $\ms{Y} \ra T$ is flat. By the principle of connectedness, this implies that $Y(p,q_t)$ is connected for each point $t\in T$. 

Now, let us fix a point $q \in B\backslash \{p\}$ such that $Y(p,q)$ is connected and consider the $1$-nodal curve $B_{pq}$. Since both $W^{r+1}_d(B)$ and $W^r_{d-2}(B)$ are empty, then by Lemma \ref{wirtcon helper 2}, the schemes $W(p,q)$ and $Y(p,q)$ have homeomorphic underlying topological spaces, hence $W(p,q)$ is connected. As described before the proof, $\wti{W}^r_d(B_{pq})$ is obtained by gluing two copies of $W(p,q)$, which implies that $\wti{W}^r_d(B_{pq})$ is connected. 
\end{proof}

\begin{proof} [Proof of Lemma \ref{wirtcon helper 1}] Say that a smooth curve $B$ of genus $g-1$ has property $\mc{E}$, if $W^{r+1}_d(B) = W^r_{d-2}(B) = \emptyset$. Note that $\mc{E}$ is an open property and depends on the integers $r,d$.
Since $\rho(g,r,d) < \min\{r+1, g+r-d\}$, then both $\rho(g-1, r+1, d)$ and $\rho(g-1,r, d-2)$ are negative. Therefore, by [ACGH, Thm.1.5, p.214], $\mc{E}$ holds for a general smooth curve of genus $g-1$. By [EH, Thm.4.5 and Sec.9], there exists a rational $g$-cuspidal curve $X$ for which $\ol{G}^r_d(X)$ is connected, reduced, and of the expected dimension. It follows that for a general $1$-cuspidal curve $B_{pp}$ of arithmetic genus $g$ the scheme $\ol{G}^r_d(B_{pp})$ is connected, reduced, and of dimension $\rho(g,r,d)$. 

For the remainder of the proof let us fix one such curve $B_{pp}$ that also has property $\mc{E}$. In particular, $W^{r+1}_d(B)=\emptyset$, which implies that  $\ol{W}^r_d(B_{pp}) = \ol{G}^r_d(B_{pp})$. By Section \ref{Presentation schemes}, there is a diagram
$$\xymatrix{ 
P_{B_{pp}}^d \ar[r]^-{\kappa} \ar[d]_-{\lambda} & \bar{J}^d_{B_{pp}} \\
J^d_B,
}$$
where $\kappa$ is proper, birational and bijective (hence, a homeomorphism on the underlying topological spaces). Let $W(p,p)$ be the scheme-theoretic preimage $\kappa^{-1}(\ol{W}^r_d(B_{pp}))$. Since $\kappa$ is birational and bijective and $\ol{W}^r_d(B_{pp})$ is connected, reduced, and of dimension $\rho(g,r,d)$, then $W(p,p)$ is connected, generically reduced, and of dimension $\rho(g,r,d)$. Note that $W(p,p)$ is non-reduced along $\kappa^{-1}(\pl \ol{W}^r_d(B_{pp}))$, see Section \ref{Presentation schemes}. 

The scheme $Y(p,p)$ is a determinantal locus whose expected dimension is $\rho(g,r,d)$, and therefore, $\dim Y(p,p) \ge \rho(g,r,d)$, if $Y(p,p)$ is nonempty. In the paragraph below, we shall show that $Y(p,p)$ is the set-theoretic image of $W(p,p)$ under $\lambda$. This implies that $Y(p,p)$ is connected and of dimension $\rho(g,r,d)$. Moreover, being a determinantal locus of the expected dimension, $Y(p,p)$ is Cohen-Macaulay, [Fu, Thm.14.3c, p.250], and therefore, has no embedded components. Thus, to show that $Y(p,p)$ is reduced, it remains to prove that $Y(p,p)$ is generically reduced. In [EH, Sec.4], it is shown that the $G^r_d$ of a cuspidal curve $X$ is closely related to a certain determinantal locus in a Grassmann bundle over the scheme of linear series of the normalization of $X$. In the case of 1-cuspidal curves this determinantal locus is the scheme $Y(p,p)$. In particular, using [EH, Thm.4.1, p.388 and Remark on p.389], it is easy to see that the morphism $G^r_d(B_{pp}) \ra Y(p,p)$, given by $L\mapsto \nu^*L$, is birational. Since $G^r_d(B_{pp})$ is reduced, this shows that $Y(p,p)$ is generically reduced. 

To complete the proof, let us check that $\lambda(W(p,p))=Y(p,p)$. Let $\mrm{sk}$ be the sky-scraper sheaf on $B_{pp}$ supported at the cusp with fiber $\mb{C}$. If $L \in \ol{W}^r_d(B_{pp})$ and $L$ is invertible, then $\kappa^{-1}(L)$ is a single reduced point $(L \hookrightarrow \nu_*\nu^*L) \in W(p,p)$, whose image in $J^d_B$ is $\nu^*L$. Consider the short exact sequence $0 \ra L \ra \nu_*\nu^*L \ra \mrm{sk} \ra 0$ and the associated long exact sequence
$$\xymatrix{0 \ar[r] & H^0(L) \ar[r] & H^0(\nu^*L) \ar[r]^-\beta & \mb{C} \ar[r] & \cdots}$$
Since $W^{r+1}_d(B) =\emptyset$, then $h^0(\nu^*L)=r+1$ and $\beta$ is the zero map. Therefore, the linear system $|L|$ pulls-back to the \emph{complete} linear system $|\nu^*L|$,    which implies that $h^0(\nu^*(L)(-2p))=r$. Hence, $\nu^*L \in Y(p,p)$. If $L \in \pl \ol{W}^r_d(B_{pp})$, let $M := (\nu^*(L)/torsion)\otimes \mc{O}_B(p)$. In this case we have $L \simeq \nu_*(M(-p))$, [Al, Lem.1.5], and there is a natural presentation $\nu_*(M(-p)) \hookrightarrow \nu_*M$. The fiber $\kappa^{-1}(L)$ is the point $(\nu_*(M(-p)) \hookrightarrow \nu_*M) \in W(p,p)$ with multiplicity 2 (see Section \ref{Presentation schemes}), whose image in $J^d_B$ is $M$. Since $h^0(L)\ge r+1$ and $W^{r+1}_d(B) = \emptyset$, the inclusion $L\simeq \nu_*(M(-p))\hookrightarrow \nu_*M$ induces an isomorphism $H^0(M(-p)) \simeq H^0(M)$. Therefore, $h^0(M)= r+1$ and $h^0(M(-2p)) \ge r$, hence $M \in Y(p,p)$. 
\end{proof}

\begin{proof} [Proof of Lemma \ref{wirtcon helper 2}] Let $\mrm{sk}$ be the sky-scraper sheaf on $B_{pq}$ supported at the node with fiber $\mb{C}$. By Section \ref{Presentation schemes}, the presentation scheme $P_{B_{pq}}^d$ fits into the diagram
$$\xymatrix{ 
P_{B_{pq}}^d \ar[r]^-{\kappa} \ar[d]_-\lambda & \bar{J}^d_{B_{pq}} \\
J^d_B,
}$$
where $\lambda$ is a $\bP^1$-bundle and $\kappa$ is the normalization morphism. To prove the lemma, we shall show that $\lambda$ restricts to a bijective morphism $\bar{\lambda}\colon W(p,q) \ra Y(p,q)$. Since $\bar{\lambda}$ is also proper, this will imply that $\bar{\lambda}$ induces a homeomorphism on the underlying topological spaces.

First, let us check that $\lambda(W(p,q))=Y(p,q)$. Recall from Section \ref{Presentation schemes} that $\lambda$ sends a presentation $(L \hookrightarrow \nu_*M)$ to $M \in J^d_B$. If $L \in \ol{W}^r_d(B_{pq})$ and $L$ is invertible, then $\kappa^{-1}(L)$ is a single reduced point $(L \hookrightarrow \nu_*\nu^*L) \in W(p,q)$. Consider the short exact sequence $0 \ra L \ra \nu_*\nu^*L \ra \mrm{sk} \ra 0$ and the associated long exact sequence
$$\xymatrix{0 \ar[r] & H^0(L) \ar[r] & H^0(\nu^*L) \ar[r]^-\beta & \mb{C} \ar[r] & \cdots}$$
Since $W^{r+1}_d(B) =\emptyset$, then $h^0(\nu^*L)=r+1$ and $\beta$ is the zero map. 
Therefore, the linear system $|L|$ pulls-back to the \emph{complete} linear system $|\nu^*L|$,    which implies that $h^0(\nu^*(L)(-p-q))=r$, see also [Ca09, Rem.2.2.1, p.1397]. Hence, $\nu^*L \in Y(p,q)$.
If $L \in \pl \ol{W}^r_d(B_{pq})$, then $\kappa^{-1}(L)$ consists of two reduced points: $(\nu_*M(-p) \hookrightarrow \nu_*M)$ and $(\nu_*M^\p(-q) \hookrightarrow \nu_*M^\p)$, where $M = (\nu^*(L)/torsion)\otimes \mc{O}_B(p)$ and $M^\p = M\otimes \mc{O}_B(q-p)$, see Section \ref{Presentation schemes}. Note that $L \simeq \nu_*M(-p)\simeq \nu_*M^\p(-q)$, which implies that $h^0(M), h^0(M^\p) \ge r+1$ and $h^0(M(-p-q)), h^0(M^\p(-p-q)) \ge r$. Hence, $M, M^\p \in Y(p,q)$. This shows that $\lambda(W(p,q))=Y(p,q)$.

Second, let us show that $\bar{\lambda}$ is bijective. Since $W^{r+1}_d(B)$ and $W^r_{d-2}(B)$ are empty, then for each $M \in Y(p,q)$, we have $h^0(M) =r+1$ and $h^0(M(-p-q))=r$. As a consequence, there is a \emph{unique} morphism $h\colon \nu_\ast M \thra \mrm{sk}$, which induces the zero map on global sections (if $p$ and $q$ are not base points of $M$, this also follows from [Ca09, Lem.5.1.3(2), p.1420]). The sheaf $L_M :=\ker(h)$ has $h^0(L_M)=r+1$, and therefore, $L_M \in \ol{W}^r_d(B_{pq})$. From the description of $\bar{\lambda}$ given above, we see that the assignment $M\mapsto (L_M \hookrightarrow \nu_*M)$ is a set-theoretic inverse of $\bar{\lambda}$, which shows that $\bar{\lambda}$ is bijective. 
\end{proof}

\begin{remalgeq}\label{remalgeq} \emph{The special subvarieties are not always algebraically equivalent, as is shown in the following example. In the case of a trigonal curve $C$, the special subvarieties $V_0$ and $V_1$ associated to a $g^1_3$ on $C$ are related by $V_1=(-1)^\ast V_0$ (after an appropriate translation) and $P$ is isomorphic to the Jacobian of $V_0$, which has a $g^1_4$, see [Be82, p.360 and p.366; Re]. Since $\rho(g,1,4)=6-g$, then all curves of genus at most $6$ admit a $g^1_4$. It follows from [Ce] and [Re] that $V_0$ and $V_1$ are not algebraically equivalent on a generic ``trigonal" Prym of dimension $3 \le p=g-1 \le 6$. Note that $\rho(g,1,3)=4-g\le 0$, when $p\ge 3$.}
\end{remalgeq}

\section{Classes of special subvarieties in the Chow ring modulo algebraic equivalence.} \label{SecOnClasses}

In this section we assume that $\pi\colon \ti{C} \ra C$ is \'etale and recall that in this case $\dim P= p=g-1$. Let us fix two integers $r$ and $d$ such that $0 < 2r < d < 2g$. Throughout this section the letters $n$ and $m$ will denote $r$-tuples of non-negative integers $(n_1, \ldots, n_r)$ and $(m_1, \ldots, m_r)$, respectively.   As in Section \ref{Special subvarieties}, $\ti{\varphi}_d\colon \ti{C}_d \ra \ti{J}$, $\varphi_d\colon C_d \ra J$ are abelian sum mappings and $\pi_d\colon \ti{C}_d \ra C_d$ is the map induced by $\pi\colon \ti{C} \ra C$. Also, let $|n| := \sum_{j=1}^r n_j$ and $\mu_{n} := \prod_{j=1}^r \frac{(-1)^{n_j-1}}{n_j}$.

For any pair $n,m$ of $r$-tuples with $1 \le n_1 \le \ldots \le n_r$, $\sum_{j=1}^r n_j \le d$, and $ 0\le m_j \le n_j/2$ for all $j$, define the following numbers: 

1. $\nu_{n,m}:$ For each $\ell \ge 1$ let $q(\ell)$ be the number of $n_j$'s that are equal to $\ell$ and suppose that $n_{j_1}=n_{j_2}=\ldots =n_{j_{q(\ell)}}=\ell$. If $q(\ell)$ is not zero, let $p(\ell, n,m)$ be the number of permutations of the ordered $q(\ell)$-tuple $(m_{j_1},m_{j_2},\ldots, m_{j_{q(\ell)}})$, and otherwise let $p(\ell, n,m)=1$. Then we set $$\nu_{n,m}:=\prod_{\ell=1}^{d-r+1}\frac{1}{ p(\ell, n,m)}.$$ Note that if $r=1$, then $\nu_{n,m}=1$ for all $n,m$. If $r=2$, then $\nu_{n,m}=1/2$ if $n_1=n_2$ and $m_1\ne m_2$, and $\nu_{n,m}=1$ otherwise.

2. $\lambda_{n,m}:$ This is the number   $$\lambda_{n,m}:=2^{d-|n|}\cdot \mu_{n}\cdot \nu_{n,m} \cdot \binom{d}{|n|}\binom{n_1}{m_1}\cdots \binom{n_r}{m_r}.$$

3. $d_{n, m}:$ Let $e_1, \ldots, e_k$ count the number of repeats in the sequence of pairs $(n_1, m_1), \ldots, (n_r, m_r)$. For example, if the sequence of pairs is 
$(1,2), (1, 2), (2, 5), (2, 3), (2, 3), (2, 3), (7,5), (3,3), (3, 3), (3,3),$ then the associated sequence of repeats is $2, 1, 3, 1, 3$. Let us define $$d_{n, m} := e_1!e_2!\cdots e_k!.$$

\begin{classes} \label{classes}
Let $0 < 2r < d < 2g$ and let $V = V_0 \cup V_1$ be the union of special subvarieties of $P$ associated to a complete and base point free $g^r_d$ on $C$. The component of the class $[V]$ in $A^{p-r}(P)_{(t)}$ is given by the formula $$[V]_{(t)} = c_{t,r,d}\big([\psi(\ti{C})]^{\ast r}\big)_{(t)}$$
where $$c_{t,r,d}:=2^{-2r-t} \sum_{n}\sum_{m\le \frac{n}{2}} \frac{\lambda_{n,m}}{d_{n,m}} \prod_{j=1}^r(n_j-2m_j)^{t+2} ,$$ 
the outer sum is taken over the choices of $r$-tuples $n=(n_1, \ldots, n_r)$ of integers with $1 \le n_1 \le \ldots \le n_r$ and $\sum_{j=1}^r n_j \le d$, the inner sum is taken over the choices of $r$-tuples $m=(m_1, \ldots, m_r)$ of integers with $ 0\le m_j \le \frac{n_j}{2}$ for all $j$.
\end{classes}
\begin{proof} 
Let $G_d$ denote the complete and base-point-free $g^r_d$ on $C$, considered as a subvariety of $C_d$ isomorphic to $\bP^r$.  Given an $r$-tuple $n$ of positive integers, consider the generalized diagonal
$$\delta_{n}=\Big\{ n_1 x_1 + n_2 x_2+\cdots+n_r x_r  \, | \, x_1, \ldots, x_r \in C\Big\}$$
in the $|n|$-fold symmetric product of $C$. 
Let $D \in G_d$ be a fixed effective divisor, whose support consists of $d$ distinct points. In [He, Thm.3, p.888] we may find the following formula for the class $[G_d] \in \mrm{CH}^{d-r}(C_d)$:
$$[G_d] = \sum_{n, \underline{o}_s}\mu_{n}[\delta_{n} + o_1+\cdots+o_s],$$    
where $s = s(n) := d - |n| \ge 0$ and the sum is taken over all $r$-tuples $n$ with $1 \le n_1 \le \ldots \le n_r$ and the choices of (unordered) sums $\underline{o}_s:=o_1+\cdots+o_s$ of pairwise distinct points in the support of the fixed divisor $D$.

In order to compute $\pi_d^\ast [G_d]$, for each pair $n,m$ of $r$-tuples with $m_j \le n_j$ for all $j$, we introduce the \emph{modified generalized diagonals} 
$$\ti{\delta}_{n,m}:=\Big\{\sum_{j=1}^r m_j\iota(\ti{x}_j)+(n_j-m_j)\ti{x}_j \, |\,  \ti{x}_1, \ldots, \ti{x}_r \in \ti{C}\Big\},$$
considered as subvarieties of the $|n|$-fold symmetric product of $\ti{C}$. We may check that 
$$\pi_d^\ast [\delta_{n} + o_1+\cdots+o_s]= \sum_{m, \underline{\ti{u}}_s} \nu_{n,m} \binom{n_1}{m_1}\cdots \binom{n_r}{m_r}\big[\ti{\delta}_{n,m}+\ti{u}_1+\cdots + \ti{u}_s\big], $$
where the sum is taken over all $r$-tuples $m$ with 
$0 \le m_j \le  \frac{n_j}{2}$ for all  $j$ 
and all (unordered) sums $\underline{\ti{u}}_s=\ti{u}_1+ \cdots +\ti{u}_s$ with $\ti{u}_j \in \pi^{-1}(o_j) = \{\ti{o}_j, \iota(\ti{o}_j)\}$. To explain the  number $\nu_{n,m}$, note that two ordered indexing pairs $n,m$ and $n,m^\p$ with $m \ne m^\p$ may label the same generalized diagonal. The number $\nu_{n,m}$ is the adjustment for this redundancy. Passing to $A(\ti{C}_d)$, the formula for the pull-back of a generalized diagonal becomes 
$$\pi_d^\ast [\delta_{n} + o_1+\cdots+o_s]= \sum_{m \le \frac{n}{2}} 2^{d-|n|}\nu_{n,m} \binom{n_1}{m_1}\cdots \binom{n_r}{m_r}\big[\ti{\delta}_{n,m}+(d-|n|)\ti{o}\big]$$
and we get:
\begin{equation}\label{pidGd}
\pi_d^\ast [G_d] = \sum_{n} \sum_{m \le \frac{n}{2}} \lambda_{n,m} \big[\ti{\delta}_{n,m}+(d-|n|)\ti{o}\big].
\end{equation}

The abelian sum mapping $\ti{\varphi}_d$ maps the subvariety $\ti{\delta}_{n,m}+(d-|n|)\ti{o}$ of $\ti{C}_d$ bijectively onto a translate of the variety $(m_1\iota+n_1-m_1)(\ti{C})+\cdots+ (m_r\iota+n_r-m_r)(\ti{C}),$
where $\ti{C}\subset \ti{J}$ also denotes the Abel curve and $(m_j\iota+n_j-m_j)$ is viewed as an endomorphism of $\ti{J}$. The number $d_{n, m}$ computes the degree of the addition morphism 
$$(m_1\iota+n_1-m_1)(\ti{C}) \times \cdots \times (m_r\iota+n_r-m_r)(\ti{C}) \ra \sum_{j=1}^r(m_j\iota+n_j-m_j)(\ti{C}),$$
and therefore, by the definition of Pontryagin product we have:
$$(m_1\iota+n_1-m_1)_\ast[\ti{C}] \ast \cdots \ast (m_r\iota+n_r-m_r)_\ast[\ti{C}]=d_{n,m}\Big[\sum_{j=1}^r(m_j\iota+n_j-m_j)(\ti{C})\Big],$$
which gives the formula:
\begin{equation} \label{uPhiTiDelta}
u_\ast \ti{\varphi}_{d\ast} [\ti{\delta}_{n,m}+(d-|n|)\ti{o}]=
\frac{1}{d_{n,m}}(n_1-2m_1)_\ast [\psi(\ti{C})] \ast \cdots \ast (n_r-2m_r)_\ast [\psi(\ti{C})].
\end{equation} 
Since the composition of $P \hookrightarrow \ti{J}$ with  $u\colon \ti{J}\ra P$ is multiplication by 2, then by (\ref{pidGd}) and (\ref{uPhiTiDelta}) we have the following identity in $A(P)$: 
$$2_\ast [V] =  \sum_{n}\sum_{m\le \frac{n}{2}} \frac{\lambda_{n, m}}{d_{n,m}}(n_1-2m_1)_\ast [\psi(\ti{C})] \ast \cdots \ast (n_r-2m_r)_\ast [\psi(\ti{C})].$$
We may extract the formulas for the homogeneous components of $[V]$ by recalling that $k_\ast x = k^{2(p-l)+t}x$ for $x \in A^l(P)_{(t)}$.
\end{proof} 

In general, there is no canonical way of distinguishing $V_0$ from $V_1$, and consequently, we may not extract the formulas for $[V_0]$ and $[V_1]$ from the formula for $[V]$ in a direct way. We do not know in general if $[V_i] \in \ms{T}(P, \psi(\ti{C}))$. However, in the examples of Remark \ref{remalgeq}, $[V_i]_{(1)} \ne 0$, which implies that $[V_i] \notin \ms{T}(P, \psi(\ti{C}))$, because by Theorem \ref{prymgens}, $\ms{T}(P, \psi(\ti{C}))$ is generated by elements of even Beauville degree. When $V_0$ and $V_1$ are algebraically equivalent, $[V_0]=[V_1]$ in $A(P)$,  and hence, $[V_0]=[V_1] = \frac{1}{2}[V]\in \ms{T}(P, \psi(\ti{C}))$. Also, note that even when $c_{t,r,d} \ne 0$, the class $([\psi(\ti{C})]^{*r})_{(t)}$ may itself be zero, e.g., when $t$ is odd and $r=1$, see the proof of Theorem \ref{prymgens}.

\subsection*{Example 1} \label{Example1}
When $G_d$ is a $g^1_d$ (the case $r=1$), the formula in Theorem \ref{classes} reads
\begin{equation} \label{g1d}
[V]_{(t)} =\sum_{n=1}^d\sum_{m \le \frac{n}{2}} \frac{(-1)^{n-1}2^{d-n-t-2}}{n}\binom{d}{n}\binom{n}{m} (n-2m)^{2+t}[\psi(\ti{C})]_{(t)}.
\end{equation}
In particular,  we may check that $$[V]_{(0)}=2^{d-3}[\psi(\ti{C})]_{(0)}=2^{d-2}\cdot \frac{\xi^{g-2}}{(g-2)!},$$ which corresponds to the singular cohomology class of $V$ and agrees with the formula in [Be82, Thm.1, p.364]. 

When $t=2s \ge 0$ is even, using the package {\tt ekhad} from [PWZ] for the software system {\tt maple} and the resource [OEIS], it appears that 
$$c_{2s,1,d}=\frac{(4^{s+1}-1)B_{2s+2}}{s+1}\cdot2^{d-2},$$  
where $B_m$ is the $m^{\mrm{th}}$ Bernoulli number, defined by $\frac{t}{e^t-1} = \sum_{m=0}^\infty B_m\frac{t^m}{m!}.$ In particular, the coefficient $c_{t,1,d}$ is non-zero when $t$ is even. When $t$ is odd,  a closed formula for $c_{t,1,d}$ can also be found using the package {\tt ekhad}, but the formula we obtained was very bulky, and therefore, we do not include it here. In any case, when $t$ is odd, the class $[\psi(\ti{C})]_{(t)}$ is zero.

\subsection*{Example 2}  When $G_d$ is a $g^3_7$, we may check that 
$$\begin{array}{ll}
2_\ast [V] =& \frac{821}{6}Z^{\ast 3} - 84 Z^{\ast 2} \ast 2_\ast Z + \frac{89}{6}Z^{\ast 2} \ast 3_\ast Z-\frac{7}{4}Z^{\ast 2} \ast 4_\ast Z + \frac{1}{10}Z^{\ast 2} \ast 5_\ast Z\\ 
&\\
&+  \frac{89}{8}Z \ast (2_\ast Z)^{\ast 2} - \frac{7}{3}Z\ast 2_\ast Z \ast 3_\ast Z +\frac{1}{8}Z\ast 2_\ast Z \ast 4_\ast Z +  \frac{1}{18}Z \ast (3_\ast Z)^{\ast 2}\\ 
&\\
&- \frac{7}{24} (2_\ast Z)^{\ast 3} + \frac{1}{24}(2_\ast Z)^{\ast 2}\ast 3_\ast Z,
\end{array}$$
where $Z = [\psi(\ti{C})]$. Using the identification $[\psi(\ti{C})]_{(0)}=2\cdot \frac{\xi^{g-2}}{(g-2)!}$ and [BL, Cor.16.5.8, p.538], an elementary calculation shows that $[V]_{(0)}=2\cdot \frac{\xi^{g-4}}{(g-4)!}$, which agrees with the formula in [Be82].\\

Using Remark \ref{RemNotPolynInXi}, we may deduce the following non-vanishing results.
From the first example above we know that $c_{2, 1, d} = -2^{d-4}$, which implies that $[V]_{(2)} \ne 0$, where $V$ is associated to a $g^1_d$ on a generic curve $C$ of genus $g\ge 6$.     

When $r=2$, we have verified on a computer that $c_{2, 2, d} = 2^{d-7}$ for $0 \le d \le 100$. Therefore, $[V]_{(2)} \ne 0$ at least for $4 \le d \le 100$, where $V$ is associated to a $g^2_d$ on a generic curve $C$ of genus $g \ge 6$.  

When $r = 3$, the coefficients $c_{2,r,d}$ are not always integers and do not seem to follow an obvious pattern, but they still appear to be non-zero, which has been checked on a computer for $3 \le d \le 50$. 

\section*{Acknowledgements} I would like to thank my thesis advisor Elham Izadi for suggesting the questions addressed in this article and for her generous help and support. I would like to thank Robert Varley for sharing his insights during many stimulating discussions on the topic of this article and on algebraic curves and abelian varieties in general. Also, I would like to thank Valery Alexeev, Silvia Brannetti, and Lucia Caporaso for answering my questions on Brill-Noether theory for a nodal curve. Finally, I would like to thank the referee for suggestions towards the improvement of the article.

\end{document}